\documentclass[12pt,twoside,reqno]{amsart}
\usepackage{amsmath}
\usepackage{amsfonts}
\usepackage{amssymb}
\usepackage{color}
\usepackage{mathrsfs}

\usepackage{cite}
\newcommand{\rd}{\mathrm{d}}
\usepackage[mathlines]{lineno}
\usepackage{geometry}
\usepackage{marginnote}
\usepackage{todonotes}
\allowdisplaybreaks
\newcommand{\mA}{\mathcal{A}}
\newtheorem{theorem}{Theorem}[section]
\newtheorem{corollary}[theorem]{Corollary}
\newtheorem{lemma}[theorem]{Lemma}
\newtheorem{proposition}[theorem]{Proposition}

\newtheorem{remark}[theorem]{Remark}
\allowdisplaybreaks
\numberwithin{equation}{section}
\begin{document}
\title{Prevention of Infinite-time Blowup by Slightly Super-linear Degradation in a  Keller--Segel System with Density-suppressed Motility} 
\thanks{}
\author{Yamin Xiao}
\address{School of Mathematical Sciences, Hebei Normal University, Shijiazhuang 050024, Hebei Province, P.R. China}
\email{xiaoyamin@hebtu.edu.cn}

\author{Jie Jiang}
\address{Innovation Academy for Precision Measurement Science and Technology, Chinese Academy of Sciences,  Wuhan 430071, HuBei Province, P.R. China; University of Chinese Academy of Sciences,  Beijing 100049, P.R. China}
\email{jiang@apm.ac.cn, jiang@wipm.ac.cn}


\date{\today}

\begin{abstract}
An initial-Neumann boundary value problem for a Keller--Segel system with density-suppressed motility and source terms is considered. Infinite-time blowup of the classical solution was previously observed for its source-free version when dimension $N\geq2$. In this work, we prove that with any source term involving a slightly super-linear degradation effect on the density, of a  growth order  of $s\log s$ at most, the classical solution is uniformly-in-time bounded when $N\leq3$, thus preventing the  infinite-time explosion detected in the source-free counter-part. The cornerstone of our proof lies in an improved comparison argument and a construction of an entropy inequality. 
\end{abstract}

\maketitle

%
%
\pagestyle{myheadings}
\markboth{\sc{Y. ~Xiao \& J.~Jiang}}{\sc{Prevention of Infinite-time Blowup in a Keller--Segel System}}

\section{Introduction}In the present contribution, we consider the initial-boundary value problem
\begin{subequations}\label{ks}
	\begin{align}
		& \partial_t u = \Delta\big( u \gamma(v)\big)- u f(u), \qquad &(t,x) &\in  (0,\infty) \times\Omega, \label{ks1}\\
		& \tau \partial_t v = \Delta v -v + u ,  \qquad &(t,x) &\in (0,\infty) \times\Omega, \label{ks2}\\
		& \nabla\big( u\gamma(v)\big)\cdot \mathbf{n} = \nabla v\cdot \mathbf{n} = 0, \qquad &(t,x) &\in (0,\infty) \times\partial\Omega, \label{ks3}\\
		& (u,\tau v)(0) = (u^{in},\tau v^{in}), \qquad &x &\in\Omega. \label{ks4}
	\end{align}
\end{subequations}Here, $\Omega$ is a smooth bounded domain of $\mathbb{R}^N$, $N\le 3$, and $\mathbf{n}$ denotes the outward unit normal vector field to $\partial\Omega$. $u$ and $v$ stand for the density of cells and the concentration of the signal, respectively. $\tau\geq0$ is a constant, and $uf(u)$ describes the proliferation as well as the degradation in  population. The motility $\gamma(\cdot)$ is a monotone non-increasing function of the signal, which characterizes a repressive effect of signal concentration on cellular motility.

System \eqref{ks1}--\eqref{ks2} with  a standard logistic growth term (i.e., $uf(u)=\mu(u^2-u)$ with $\mu>0$) has recently been proposed in \cite{PhysRevLett.108.198102} to model the formation of strip patterns driven by density-suppressed motility in a bacteria environment with reproduction and degradation. As experimentally observed in \cite{PhysRevLett.108.198102,2011Science}, this model correctly captures the dynamics at the propagating front where new stripes are formed. Setting $f(\cdot)\equiv0$, we note that system \eqref{ks1}--\eqref{ks2} also belongs to a general class of the so-called Keller--Segel system modeling  the directed movement of cells in response to chemical signals \cite{1971KS}. The non-increasing monotonicity of motility function corresponds to a chemo-attraction phenomenon.

\medskip
Theoretical analysis of system \eqref{ks1}--\eqref{ks2} has attract broad interest in recent years. Assuming $f\equiv0$, if $\gamma$ is bounded from above and below, and moreover $|\gamma'|$ is bounded, it was proved in \cite{TW2017} that for convex domain $\Omega\subset\mathbb{R}^2$, the classical solution exists globally and is uniformly bounded. Later, the result was improved in \cite{XiaoJiang2022} in a sense that merely under the upper and lower boundedness assumption on $\gamma$, global classical solution always exists and is bounded on any smooth bounded (not necessarily convex) domain in $\mathbb{R}^N$ ($N\geq1$) and  the requirement on boundedness of $|\gamma'|$ is removed as well. If $\gamma$ is non-increasing and tends to zero at infinity, then system \eqref{ks1}--\eqref{ks2} may become degenerate as $v\rightarrow\infty$. Under the circumstances, it was shown in \cite{FuJi2020,FuJi2021a,JiLa2021,FuSe2022a,JLZ2022,FuSe2022b,XiaoJiang2022} that classical solution always exists globally for $N\geq1$. Moreover, its dynamics is shown to be closely related to the decay rate of $\gamma$ at infinity. More precisely, when $N=2$, the solution is uniformly bounded whenever $\gamma$ decays slower than a negative exponential function, i.e.,  $\gamma$  satisfies 
\begin{equation}\label{exp}
	\liminf\limits_{s\rightarrow\infty} e^{\chi s}\gamma(s)>0
\end{equation}
 for all $\chi>0$, typical examples being $\gamma(s)=s^{-k_1}\log^{-k_2} (1+s)$ with any $k_1,k_2>0$, or $\gamma(s)=e^{-s^{\alpha}}$ with any $0<\alpha<1$. If $\gamma(s)\sim e^{-\chi s}$ as $s\rightarrow\infty$ or, in other words, there is $\chi>0$ such that \eqref{exp} if fulfilled, then the classical solution is bounded if the conserved total mass $\|u^{in}\|_{L^1(\Omega)}$ is strictly less than a threshold number $\Lambda_c=4\pi/\chi$ \cite{FuJi2021b,JLZ2022,XiaoJiang2022}.  In the higher dimensional cases $N\geq3$, uniform-in-time boundedness was established provided that $\gamma(s)\sim s^{-k}$ as $s\rightarrow\infty$, with $0<k<\frac{N}{N-2}$ \cite{JiLa2021,JLZ2022,FuSe2022b,XiaoJiang2022}.

 Boundedness results were also obtained when various source terms get involved, as a dampening effect is brought in by the degradation in population. With a choice of standard logistic-type growth, i.e., $uf(u)=\mu (u^2-u)$, it was proved in \cite{JKW2018} that when $N=2$ and $\tau>0$, with any $\gamma\searrow0$ satisfying $\lim\limits_{s\rightarrow\infty}\frac{\gamma'(s)}{\gamma(s)}<\infty$,  there exists a uniformly-in-time bounded global classical  solution to our problem \eqref{ks} whenever $\mu>0$. A similar boundedness result was established in \cite{Winkler2023IMRN} with a weakened assumption on $\gamma$ that $\gamma>0$ on $[0,\infty)$ and $|\gamma'|+|\gamma''|\in L^\infty(0,\infty)$.  In higher dimensional case $N\geq3$, for $\gamma'<0$ and $|\gamma'|<\infty$, sufficiently largeness of $\mu$ was assumed to establish the boundedness of classical solutions in \cite{WangWang2018JMP}.  When $\tau=0$ and $N=2$, the above extra assumption on $\gamma'/\gamma$ was completely removed in   \cite{FuJi2020}, and when $N\geq3$, largeness of $\mu$ was still needed in \cite{Tello2022}. 
  Besides, generalized logistic-type source term $uf(u)=\mu u^\kappa-\lambda u$ with $\mu,\kappa,\lambda>0$ was also considered in the case $\tau=0$ \cite{LyWa2023}, boundedness of global solution was verified when any of following holds: (i) $N\leq2$, $\kappa>1$, (ii) $N\geq2$, $\kappa>2$, or (iii) $N\geq3$, $\kappa=2$ and $\mu$ is sufficiently large. Recently, existence of weak solutions was also proved in \cite{DLTW2023} when $\big(f(s)\log s\big)/s\rightarrow\infty$ as $s\rightarrow\infty$, and $\sup_{s>s_0}\big(\gamma(s)+s|\gamma'(s)|^2/\gamma(s)\big)<\infty$ for all $s_0>0$ for $N\geq2$.
   
\medskip
 Aside from the above boundedness/existence results,  particular interest was paid to the specific case $\gamma(s)=e^{-\chi s}$. Note the exponential  decay rate is critical when $N=2$. With this choice of $\gamma$, system \eqref{ks1}--\eqref{ks2} with $f\equiv0$ shares some similar important features with the extensively studied minimal Keller--Segel system that reads as
 \begin{equation}
 	\begin{cases}\label{minks}
 	u_t=\Delta u-\chi\nabla \cdot(u\nabla v)\;,\\
 	\tau v_t=\Delta v-v+u\,.
 	\end{cases}
 \end{equation}
First, under the homogeneous Neumann boundary condition, they have the same steady states $(u_s, v_s)$ which satisfies
\begin{equation}
	\begin{cases}
	-\Delta v_s+v_s= \Lambda e^{\chi v_s}/\int_\Omega e^{\chi v_s}\,\rd x\,,\qquad &x \in \Omega\\
	u_s=\Lambda e^{\chi v_s}/\int_\Omega e^{\chi v_s}\,\rd x\,,\qquad &x \in \Omega\\
	\nabla v_s\cdot\mathbf{n}=0\,,\qquad &x \in \partial\Omega,
\end{cases} 
\end{equation}with some $\Lambda>0$. More importantly, they have similar entropy-like structures. Indeed, introducing the entropy functional
\begin{equation}
	E(u,v)=\int_\Omega \bigg(u\log u+\frac\chi2|\nabla v|^2+\frac\chi2 v^2-\chi uv\bigg)\,\rd x,
\end{equation}
then for a solution of \eqref{ks} with $\gamma(v)=e^{-\chi v}$ and $f\equiv0$, there holds
\begin{equation}
	\frac{\rd }{\rd t}E(u,v)(t)+\int_\Omega ue^{-\chi v}|\nabla \log u-\chi\nabla v|^2\,\rd x+\tau\chi\|v_t\|_{L^2(\Omega)}^2=0,
\end{equation}
while for solutions of the minimal Keller--Segel system \eqref{minks} with homogeneous Neumann boundary condition, there holds
\begin{equation}
	\frac{\rd }{\rd t}E(u,v)(t)+\int_\Omega u|\nabla \log u-\chi\nabla v|^2\,\rd x+\tau\chi\|v_t\|_{L^2(\Omega)}^2=0.
\end{equation}
The only difference lies in whether there exists a weighted function $e^{-\chi v}$ in the first dissipation term. Therefore, it is quite interesting to compare the dynamical behavior of their solutions. For the minimal Keller--Segel system, it is well-known that finite-time blowup can take place when $N\geq2$.  In contrast, for the system under consideration with $f\equiv 0$,  it was shown that with the same initial condition, the classical solution  becomes unbounded as time tends to infinity \cite{FuJi2020,FuJi2021a,JW2020,BLT2020,FuSe2022a}, unveiling an intrinsic time-delay effect on the population explosion. 

\medskip

The dampening effect of (generalized) logistic-type source terms has also been studied for the minimal Keller--Segel system by adding $-uf(u)$ on the right-hand side of the first equation in \eqref{minks}. In the case of standard logistic source, i.e., $uf(u)=\mu u^2-\lambda u$, it was proved that aggregation is suppressed when $\tau=0$ and $\mu>\frac{(N-2)_+}{N}\chi$ \cite{TelloWinkler2007} (here and below, $a_+=\max\{a,0\}$). If $\tau>0$, boundedness of classical solution was examined for $N\leq2$ with any  $\mu>0$ \cite{OTYM2002}, and for $N\geq3$ with $\mu>0$ being sufficiently large  \cite{Winkler2010CPDE}. In the case of generalized logistic source, on the one hand, it was shown that superlinear absorption as $uf(u)=\mu u^\kappa-\lambda u$ with $\kappa\in(1,2)$ may fail to prevent finite-time blowup when $N\geq3$ \cite{Winkler2018zamp}. On the other hand, some sub-logistic term like $uf(u)=\mu u^2/\log^\alpha(1+u)-\lambda u$ with $\lambda\in\mathbb{R}$, $\mu>0$ and $\alpha\in(0,1)$ was verified to be sufficient to suppress finite-time blowup of classical solutions in 2D \cite{Xiang2018JMP}. More recently, it was proved in \cite{Winkler2023} that under mere an assumption that $f(s)\rightarrow\infty$ as $s\rightarrow\infty$, the fully parabolic type ($\tau>0$) minimal Keller--Segel system  with any superlinear degradation term $-uf(u)$ subjected to homogeneous Neumann boundary condition permits existence of globally defined generalized nonnegative and integrable solutions in any dimensions, ruling out the emergence of persistent Dirac-type singularities.  In comparison, the current boundedness results as we mentioned before on \eqref{ks} with $\gamma(s)=e^{-\chi s}$ in the two-dimensional case  applies only for generalized logistic-type degradation with $\kappa>1$ and $\tau=0$ \cite{LyWa2023}, or standard logistic source $\kappa=2$ and $\tau>0$ \cite{JKW2018}. In higher dimensions, boundedness is achieved when  either generalized logistic-type source of super-quadratic degradation $\kappa>2$ for $\tau=0$ \cite{LyWa2023},  or standard logistic-type  source with sufficiently largeness of $\mu$  for $\tau\geq0$ \cite{WangWang2018JMP,LyWa2023} is assumed. Thus, it is still unclear  whether the generalized logistic-type source $\kappa>1$, or even any superlinear degradation satisfying $f(s)\rightarrow\infty$ as $s\rightarrow\infty$, is enough to prevent the infinite-time blowup  detected  in the corresponding source-free ( $f\equiv0$) system \eqref{ks} with $\gamma(s)=e^{-\chi s}$ in any dimensions $N\geq2$.

The main purpose of the present contribution is to give a partial affirmative answer to the above question. Indeed, we prove that when $N=2,3$, $\tau\geq0$ and $\gamma(s)=e^{-\chi s}$, any superlinear degradation satisfying that  $f(s)\rightarrow\infty$ , and $f(s)/\log s<\infty$ as $s\rightarrow\infty$, is sufficient to guarantee the boundedness of classical solutions, thus foiling the observed infinite-time blowup phenomenon in the source-free counterpart. Compared with  the finite-time blowup results of the minimal Keller--Segel system with certain sub-quadratic growth source when $N=3$ \cite{Winkler2018zamp}, our result unveils a significant new feature in the dynamical behavior of  solutions to our system \eqref{ks} .

\medskip
 To begin with, we introduce some basic assumptions and notations. Throughout this paper we use the short notation $\|\cdot\|_{p}$ for the norm $\|\cdot\|_{L^p(\Omega)}$ with  $p\in[1,\infty]$, and we use $\|\cdot\|$ to denote the $L^2$-norm for simplicity.
  For the initial condition $(u^{in}, \tau v^{in})$, we require that
\begin{equation}\label{ini}
	\begin{split}
		& \left( u^{in},\tau v^{in} \right)\in C(\bar{\Omega})\times  W^{1,\infty}(\Omega)\,, \quad u^{in}\not\equiv 0\,, \\
		& u^{in}\ge 0\,, \quad \tau v^{in}\ge 0\quad  \mbox{in } \bar\Omega\,.
	\end{split}
\end{equation} 
For the source term, we assume that $f(\cdot)\in C^{1}([0,\infty))$, and moreover
\begin{equation}\label{f1}
	\lim\limits_{s\rightarrow\infty}f(s)=\infty
\end{equation}and
\begin{equation}\label{f2}
	\limsup\limits_{s\rightarrow\infty}f(s)/\log s<\infty.
\end{equation}
A typical example is $f(s)=\mu\log^\alpha(1+s)-\lambda$ with $\mu>0$, $\lambda\in\mathbb{R}$ and $0<\alpha\leq 1$.

\bigskip
 We are now in a position to state our main outcome concerning global existence of uniformly-in-time bounded classical solutions to \eqref{ks}.
\begin{theorem}\label{Th1}
	Assume $N=2,3$, $\tau\geq0$ and $\gamma(s)=e^{-s}$. For any given initial data $(u^{in}, \tau v^{in})$ satisfying \eqref{ini}, the initial-boundary value problem \eqref{ks} has a unique  global non-negative classical solution $(u,v)\in \bigg(C\big([0,\infty)\times\bar{\Omega}\big) \cap C^{1,2}\big((0,\infty)\times\bar{\Omega}\big)\bigg)^2$, which is  uniformly-in-time bounded, i.e.,
	\begin{equation*}
		\sup_{t\ge 0}\left\{ \|u(t)\|_\infty + \|v(t)\|_{\infty} \right\} < \infty\,. 
	\end{equation*}
\end{theorem}
\begin{remark}
The proof is carried out for the fully parabolic case $\tau>0$ in detail. The parabolic-elliptic case $\tau=0$ follows from a simple modification and we will provide remarks where necessary. 
\end{remark}

The main difficulty in analysis of our problem lies in deriving a uniform-in-time upper bound of $v$, which is  overcome by an improved version of a comparison argument originating from \cite{FuJi2021a,FuJi2020}, together with establishing a uniform-in-time estimate for an entropy-like functional of the form
$\int_\Omega u\log u\,\rd x$.

To expose the main strategy of our proof, let us first briefly recall the idea to get the upper bound of $v$ via comparison argument developed in \cite{FuJi2021a} for the fully parabolic system \eqref{ks} with $f\equiv0$. Denoting $\mathcal{A}:=I-\Delta$, with $\Delta$ being the  Laplacian operator subject to homogeneous Neumann boundary condition, it was first introduced an auxiliary function defined by $w(t,x)=\mathcal{A}^{-1}[u](t,x)$, which satisfies, by taking the inverse  operator $\mathcal{A}^{-1}$ on both sides of \eqref{ks1},  a key identity that reads
\begin{equation}\label{kid}
	\partial_t w+u\gamma(v)=\mathcal{A}^{-1}[u\gamma(v)].
\end{equation}
For the sake of simplicity, we now assume $\gamma(s)=e^{-s}$, although the comparison argument in \cite{FuJi2021a} applies to general non-increasing motility functions. Since $e^{-s}\leq1$ for all $s\geq0$, we may infer by the comparison principle of elliptic equations that $\mathcal{A}^{-1}[ue^{-v}]\leq \mathcal{A}^{-1}[u]=w$, which together with the non-negativity of $ue^{-v}$ gives rise to $\partial_t w\leq w$ from \eqref{kid}. Thus, $w(t,x)$ is bounded from above. Note that if $\tau=0$, $w$ is identical to $v$ and thus $v$ is bounded in that case \cite{FuJi2020}. However, when $\tau>0$, we easily see from \eqref{ks2} that $w=v+\tau\mathcal{A}^{-1}[\partial_t v]$. Hence, a second comparison was designed in \cite{FuJi2021a}, with the help of the non-increasing monotonicity of $\gamma$, the key identity \eqref{kid} and \eqref{ks2}, deriving that $\mathcal{L}[v]\leq \mathcal{L}[w-\tau e^{-v}+C]$, with $\mathcal{L}:=\tau\partial_t+\mathcal{A}$ and some $C>0$ depending only on initial data. Then, comparison principle of heat equations implies that $v\leq w+C$, yielding an upper bound for $v$ finally. 

\medskip
When a source term $-uf(u)$ involving super-linear degradation is considered in \eqref{ks},  one can easily deduce a similar key identity as before in the form
\begin{equation}\label{kid0}
	\partial_t w+u\gamma(v)+\mathcal{A}^{-1}[uf(u)]=\mathcal{A}^{-1}[u\gamma(v)].
\end{equation}
The super-linearity of $uf(u)$ together with Young's inequality will lead us to a differential inequality $\partial_t w+\varepsilon_0 w\leq C$, with some $\varepsilon_0,C>0$ independent of time, and thus uniform boundedness of $w$ follows. Note that  $v=w$ if $\tau=0$ and hence $v$ is uniformly bounded in that case \cite{FuJi2020,LyWa2023}. 
 However, if $\tau>0$, the crucial second comparison between $v$ and $w$ becomes tricky due to the presence of the new non-local term $\mathcal{A}^{-1}[uf(u)]$. Indeed, in the same manner as before, we deduce that $\mathcal{L}[v]\leq \mathcal{L}[w-\tau e^{-v}+C]+\tau \mathcal{A}^{-1}[uf(u)]$, with the upper boundedness of the non-local term being unknown in general. In order to tackle this issue, in the present contribution, we introduce a new second auxiliary function $h(t,x)$, which satisfies that $\mathcal{L}[h]=\mathcal{A}^{-1}[uf(u)]$ with some given initial datum. It now allows us to apply comparison principle of heat equations to deduce  that $v\leq w+\tau h+C$. As a result, it suffices to derive the $L^\infty$-boundedness of $h$. To this aim, we further exploit the nice entropy structure of \eqref{ks} with the particular choice of $\gamma(s)=e^{-s}$, which provides a uniform estimates involving $\|u\log u\|_1$ by delicate calculations. Here, the sub-logarithmic growth of $f$ \eqref{f2}, and the one-sided control $v\leq w+\tau h+C$ are both crucial to our analysis.  Then, $L^1$-boundedness of $uf(u)$ follows since $sf(s)\leq Cs\log s+C$ due to \eqref{f2} again, and standard regularity theory of heat equation ensures a time-independent $L^\infty$-bound for $h$ for $N\leq3$ eventually. Once the boundedness of $v$ is obtained, we next show the uniform boundedness of $u$ by standard energy method when $N=2$, and by semi-group theory when $N=3$, respectively. In the latter case,  $L^1$-boundedness of $uf(u)$ is needed again to derive a H\"older continuity of $v$, which enables us to further apply the semi-group theory.

\medskip
The paper is organized as follows. In Section 2, we introduce the result on  existence and uniqueness of local classical solutions,  and some useful lemmas. In Section 3, we establish the uniform upper bound for $v$ by the improved comparison argument together with energy method. Boundedness of $u$ is established in Section 4 for $N=2$, and  in the last section  for $N=3$, respectively.

\section{Preliminaries}
In this section, we recall some useful results. First, since $f\in C[0,\infty)$ and  assumption \eqref{f1} holds, for any $a_1>0$ there is $b_1\geq0$ depending on $a_1$ and $f$ such that
\begin{equation}\label{f1a}
	sf(s)\geq a_1s-b_1,\qquad\text{for all}\;s\geq0,
\end{equation}which in particular, implies that there is $b_0\geq0$ being fixed throughout this paper (taking $a_1=1$ in \eqref{f1a}) such that
\begin{equation}\label{f1b}
	sf(s)+b_0\geq 0,\qquad\text{for all}\;s\geq0.
\end{equation}
Moreover, by \eqref{f2}, there are $a_2, b_2>0$ such that
\begin{equation}\label{f2a}
	sf(s)\leq a_2s\log s+b_2,\qquad\text{for all}\;s\geq0.
\end{equation}

Next, we show existence and some basic properties of the local classical solutions.
\begin{proposition}\label{prop2.1}
	Assume that $(u^{in}, \tau v^{in})$ satisfy \eqref{ini} with $m=\|u^{in}\|_1>0$. Then the initial-boundary value problem~\eqref{ks} has a unique non-negative classical solution 
	\begin{equation*}
		(u,v)\in \bigg(C\big([0,T_{\mathrm{max}})\times\bar{\Omega}\big) \cap C^{1,2}\big((0,T_{\mathrm{max}})\times\bar{\Omega}\big)\bigg)^2,
	\end{equation*} 
	defined on a maximal time interval $[0,T_{\mathrm{max}})$ with $T_{\mathrm{max}}\in (0,\infty]$. 
	If $T_{\mathrm{max}}<\infty$, then
	\begin{equation}
		\lim_{t\to T_{\mathrm{max}}} \|u(t)\|_\infty = \infty. \label{gec}
	\end{equation}

	In addition, for arbitrary $\varepsilon>0$, there is $C(\varepsilon)>0$ depending only on $\varepsilon,$ $u^{in}$, $\Omega$ and $f$ such that
\begin{equation}
	\|u(t)\|_1+\frac12\int_0^t e^{\varepsilon(s-t)}\int_\Omega (uf(u)+b_0)\,\rd x\rd s\leq \|u^{in}\|_1+C(\varepsilon), \label{u1a}\qquad\text{for all}\; t\in(0,T_{\mathrm{max}}),
\end{equation}and for any $0<s<t<T_{\mathrm{max}}$,
\begin{equation}\label{u1b}
	\int_s^t\int_\Omega  (uf(u)+b_0)\,\rd x\rd \sigma\leq 2b_0(t-s)+2\|u^{in}\|_1+4b_0|\Omega|.
\end{equation}	Besides, there is $C>0$ depending on $\tau$, $f$, $\Omega$ and the initial data such that
\begin{equation}\label{vl1}
	\sup\limits_{t\in(0,T_{\mathrm{max}})}\|v(t)\|_1\leq C.
\end{equation}
\end{proposition}
\begin{proof}Existence and uniqueness of local classical solutions mainly follows from the theory developed by Amann in \cite{Aman1988, Aman1989, Aman1990, Aman1993}, see e.g., \cite{AhnYoon2019,JKW2018,JLZ2022}, while the non-negativity of solutions follows from the comparison principle of heat equations.
	
	An integration of \eqref{ks1} over $\Omega$ yields that
\begin{equation}\label{u1}
	\frac{\rd}{\rd t}\int_\Omega u\,\rd x +\int_\Omega uf(u)\,\rd x=0.
\end{equation}
Choosing $a_1=2\varepsilon$ in \eqref{f1a} with any $\varepsilon>0$, we infer that
\begin{equation*}
	\frac{\rd}{\rd t}\int_\Omega u\,\rd x +\varepsilon\int_\Omega u\,\rd x+\frac12\int_\Omega(uf(u)+b_0)\,\rd x  \leq (b_1(\varepsilon)+b_0)|\Omega|/2.
\end{equation*}In particular, take $a_1=2\varepsilon=1$, then $b_1=b_0$ and  we get 
\begin{equation*}
	\frac{\rd}{\rd t}\int_\Omega u\,\rd x +\frac12\int_\Omega u\,\rd x+\frac12\int_\Omega(uf(u)+b_0)\,\rd x  \leq b_0|\Omega|.
\end{equation*}Then, solving the above two inequality, we obtain \eqref{u1a} with $C(\varepsilon)=\frac{(b_1(\varepsilon)+b_0)|\Omega|}{2\varepsilon}$, and \eqref{u1b}, respectively.

On the other hand, an integration of \eqref{ks2} over $\Omega$ yields that
\begin{equation*}
	\tau\frac{\rd}{\rd t}\int_\Omega v\,\rd x+\int_\Omega v\,\rd x=\int_\Omega u\,\rd x.
\end{equation*}Then \eqref{vl1} follows from the above inequality and the uniform $L^1$-boundedness of $u$.
 This completes the proof.
\end{proof}
Now, we recall the following   result,  see \cite[Proposition~(9.2)]{Aman1983}, \cite[Lemme~3.17]{BeBo1999}, or \cite[Lemma~2.2]{AhnYoon2019}.
\begin{lemma}\label{lm2} 
	Let $g\in  L^1(\Omega)$. For any $1\leq q< \frac{N}{(N-2)_+}$, there exists a positive constant $C_q$ depending only on $\Omega$ and $q$ such that the solution $z \in W^{1,1}(\Omega)$ to
	\begin{equation}\label{helm}
		\begin{cases}
			-\Delta z+ z=g,\qquad x\in\Omega\,,\\
			\nabla z\cdot \mathbf{n}=0\,,\qquad x\in\partial\Omega\,,
		\end{cases}
	\end{equation}
	satisfies
	\begin{equation*}
		\|z\|_q \leq C_q \|g\|_1\,.
	\end{equation*}
\end{lemma}

Finally, we  introduce the following Gronwall-type inequality, which is a variant form of \cite[Lemma 1.3]{PataZelik07}.
\begin{lemma} \label{lemgron}Let $\Lambda:\;[0,\infty)\rightarrow[0,\infty)$ be an absolutely continuous function satisfying
	\begin{equation}\label{gron0}
		\frac{\rd}{\rd t}\Lambda(t)+2\lambda\Lambda(t)\leq \alpha(t)\Lambda(t)+\beta(t),
	\end{equation}
where $\lambda>0$, $\alpha(t)\geq0$ and $\beta(t)\geq0$. Moreover, $\int_s^t\alpha(\sigma)\,\rd \sigma\leq \lambda(t-s)+m$, for all $t\geq s\geq0$ with some constant $m\geq0$, and 
\begin{equation}
	\int_0^t\beta(s)e^{\lambda(s-t)}\,\rd s\leq K_\beta.
\end{equation}
   Then, for all $t\geq0$, there holds
\begin{equation}
	\Lambda(t)\leq \Lambda(0)e^{m-\lambda t}+e^{m}K_\beta.
\end{equation}
\end{lemma}
\begin{proof}
	The proof resembles that of the classical Gronwall inequality. We provide the detail here for reader's convenience. First, we deduce from \eqref{gron0} that for all $\sigma>0$,
\begin{equation}
	\frac{\rd}{\rd \sigma}\bigg(\Lambda(\sigma)\exp\,\{2\lambda \sigma-\int_0^{\sigma}\alpha(\theta)\,\rd \theta\}\bigg)	\leq \beta(\sigma)\exp\,\{2\lambda \sigma-\int_0^{\sigma}\alpha(\theta)\,\rd \theta\}.
\end{equation}
It follows from an integration on $[0,t]$ of the above ODI that
\begin{align*}
\Lambda(t)\leq& \Lambda(0)\exp\,\{-2\lambda t +\int_0^{t}\alpha(\theta)\,\rd \theta\}+\int_0^{t}\beta(\sigma)\exp\,\{2\lambda (\sigma-t)+\int_\sigma^{t}\alpha(\theta)\,\rd \theta\}\,\rd \sigma\\
\leq& \Lambda(0)\exp\,\{-\lambda t +m\}+\int_0^{t}\beta(\sigma)\exp\,\{\lambda (\sigma-t)+m\}\,\rd \sigma\\
\leq & \Lambda(0)e^{m-\lambda t}+e^{m}K_\beta.
\end{align*}This completes the proof.	
\end{proof}
\section{Uniform-in-time boundedness of $v$}

In this section, we use an improved comparison argument to derive a uniform-in-time upper bound for $v$. Two auxiliary functions are introduced to achieve the goal. To begin with,
we introduce the operator $\mathcal{A}$ on $L^2(\Omega)$ defined by
\begin{equation}
	\mathrm{dom}(\mathcal{A}) \triangleq \{ z \in H^2(\Omega)\ :\ \nabla z\cdot \nu = 0 \;\text{ on }\; \partial\Omega\}\,, \qquad  \mathcal{A}z \triangleq - \Delta z +  z\,, \quad z\in \mathrm{dom}(\mathcal{A})\,. \label{y2}
\end{equation} 
We recall that $\mathcal{A}$ generates an analytic semi-group on $L^p(\Omega)$ and is invertible on $L^p(\Omega)$ for all $p\in (1,\infty)$. We then set 
\begin{equation}\label{defw}
	w(t) \triangleq \mathcal{A}^{-1}[u(t)]\ge 0\,,  \qquad t\in [0,T_{\mathrm{max}})\,,
\end{equation}
the non-negativity of $w$ being a consequence of that of $u$ and the comparison principle of elliptic equations. Thanks to the time continuity of $u$, 
\begin{equation*}
 w^{in}\triangleq w(0)=\mA^{-1}[u^{in}]\,,
\end{equation*} 
and it follows from the regularity assumption \eqref{ini} on the initial conditions that $w^{in}$ belongs to $W^{2,p}(\Omega)$ for any $p\in(1,\infty)$.

\begin{lemma}\label{wbound} The following key identity holds
\begin{equation}\label{keyid}
	\partial_t w+ue^{-v}+\mathcal{A}^{-1}[uf(u)]=\mathcal{A}^{-1}[ue^{-v}]\,,\qquad \text{for all}\;(t,x)\in(0,T_{\mathrm{max}})\times\Omega\,.
\end{equation} In  addition, there exists a constant $w^*>0$ depending only on $f$, and the initial data, such that for all $(t,x)\in[0,T_{\mathrm{max}})\times\Omega$
\begin{equation}\label{comA}
	w(t,x)\leq w^*.
\end{equation}
\end{lemma}
\begin{proof}
First, the identity \eqref{keyid} simply follows by taking $\mathcal{A}^{-1}$ on both sides of \eqref{ks1}. Then, since $e^{-v}\leq 1$, an application of comparison principle of elliptic equations gives
\begin{equation*}
	\mathcal{A}^{-1}[ue^{-v}](t,x)\leq \mathcal{A}^{-1}[u](t,x)= w(t,x)\,,\qquad \text{for all}\;(t,x)\in(0,T_{\mathrm{max}})\times\Omega\,.
\end{equation*}
On the other hand, according to \eqref{f2a},
\begin{equation}
	uf(u)\geq a_1u-b_1,\qquad\text{for all}\; (t,x)\in (0,T_{\mathrm{max}})\times\Omega.
\end{equation}
Then by comparison principle of elliptic equations again, we deduce that
\begin{equation*}
	\mathcal{A}^{-1}[uf(u)](t,x)\geq \mathcal{A}^{-1}[a_1u-b_1](t,x)=a_1 w(t,x)-b_1\,,\qquad \text{for all}\;(t,x)\in(0,T_{\mathrm{max}})\times\Omega\,.
\end{equation*}
which by picking $a_1\geq 2$ together with \eqref{keyid}   yields that
\begin{equation}\label{kineq1}
	\partial_t w+ue^{-v}+ w\leq b_1\,,\qquad \text{for all}\;(t,x)\in(0,T_{\mathrm{max}})\times\Omega\,.
\end{equation}
Then we finish the proof by the non-negativity of $ue^{-v}$  and solving the above inequality via standard ODI  technique.
\end{proof}

\begin{lemma}\label{dualest}
There is $C>0$ depending on $f$ and the initial data such that for all $t\in(0,T_{\mathrm{max}})$, 
\begin{equation}\label{dual0}
	\|w\|_{H^1}^2+\int_0^{t}e^{2(s-t)}\int_\Omega u^2e^{-v}\,\rd x\rd s\leq C.
\end{equation}
\end{lemma}
\begin{proof} Recalling that $w-\Delta w=u\geq0$, we multiply \eqref{kineq1} by $u$ and integrate by part to get that
\begin{equation}\label{dual1}
	\frac12\frac{\rd}{\rd t}\int_\Omega(|\nabla w|^2+w^2)\,\rd x+\int_\Omega u^2e^{-v}\,\rd x+\int_\Omega(|\nabla w|^2+w^2)\,\rd x \leq b_1\int_\Omega u\,\rd x\leq C\,,
\end{equation}
which gives rise to the desired estimates through an integration of the above ODI with respect to time.
\end{proof}
\begin{remark}
	Notice that $v=w$ if $\tau=0$. Then same results in Lemma \ref{wbound} and Lemma \ref{dualest} hold   for $v$ when $\tau=0$. Since uniform boundedness of $v$ is already known, one can skip the rest parts of this section  when $\tau=0$.
\end{remark}
\begin{lemma}\label{vh1}
 There is $C>0$ depending on $\tau,$ $f$ and the initial data such that 
\begin{equation}\label{h1}
\sup\limits_{t\in[0,T_{\mathrm{max}})}	\|v\|\leq C.
\end{equation}
\end{lemma}
\begin{proof}
	From \eqref{ks2},  \eqref{vl1}, \eqref{defw},  \eqref{comA} and  \eqref{dual0}, we infer that
	\begin{align*}
		\frac\tau2\frac{\rd}{\rd t}\|v\|^2+\int_\Omega (|\nabla v|^2+v^2)\,\rd x=&\int_\Omega uv\,\rd x
		=\int_\Omega (w-\Delta w)v\,\rd x\\
		=& \int_\Omega wv\,\rd x+\int_\Omega\nabla w\cdot\nabla v\,\rd x\\
		\leq &\|w\|_\infty\|v\|_1+\frac12\|\nabla w\|^2+\frac12\|\nabla v\|^2\\
		\leq&C+\frac12\|\nabla v\|^2,
	\end{align*}which yields that
	\begin{equation}\label{ener2}
		\frac\tau2\frac{\rd}{\rd t}\|v\|^2+\frac12\int_\Omega (|\nabla v|^2+v^2)\,\rd x\leq C.
	\end{equation}This completes the proof by solving the above inequality.
\end{proof}
\medskip

Now, we introduce the second auxiliary function $h(t,x)$ which is a solution to the following initial-boundary problem:
	\begin{equation}
	\begin{cases}\label{heqn}
		\tau \partial_t h-\Delta h+h=\mathcal{A}^{-1} [uf(u)+b_0],\qquad &(t,x) \in  (0,T_{\mathrm{max}}) \times\Omega,\\
		\nabla h\cdot\mathbf{n}=0,\qquad &(t,x) \in  (0,T_{\mathrm{max}}) \times\partial\Omega,\\
		h(0,x)=0 \qquad & x \in \Omega.
	\end{cases}
\end{equation}According to standard regularity theories of elliptic and parabolic equations, we have $h(t,x)\in C^{1,2}([0,T_{\mathrm{max}})\times\bar{\Omega})$, and moreover, $h(t,x)\geq0$ by maximum principles since $\mathcal{A}^{-1}[uf(u)+b_0]\geq0$ due to \eqref{f1b}. In addition, we have the following estimates.
\begin{lemma}
For all $(t,x)\in(0,T_{\mathrm{max}})\times\Omega$, there holds
	\begin{equation}\label{h1}
	\|h\|_1\leq C
	\end{equation}with $C>0$ depending on $\|u^{in}\|_1$, $\tau$ and $f$.
	
In addition, when $N\leq 3,$ there is  a constant $M>0$ depending only on $f$ and $\Omega$, such that  for any $\lambda_1>0$ 
\begin{align}\label{h2}
	\frac{\tau}{2}\frac{\rd}{\rd t}&\int_\Omega(|\nabla h|^2+h^2)\,\rd x+\lambda_1\int_\Omega(|\nabla h|^2+h^2)\,\rd x+\int_\Omega \left(uf(u)+b_0\right)h \,\rd x+\frac14\|h-\Delta h\|^2\nonumber\\
	&\leq 	M\left(\int_\Omega(u\log u+e^{-1}+1)\,\rd x\right)\left(\int_\Omega (uf(u)+b_0)\,\rd x\right)+C(\lambda_1)\,,\quad\text{for all}\; t\in(0,T_{\mathrm{max}})\,,
\end{align}where $C(\lambda_1)>0$ is a constant depending on $\lambda_1$, $\tau$, $f$ and $\|u^{in}\|_1$.
\end{lemma}
\begin{proof}
On the one hand,	integrating the equation for $h$ over $\Omega$ yields that
	\begin{equation*}
		\frac{\rd}{\rd t}\int_\Omega h\,\rd x+\frac{1}{\tau}\int_\Omega h\,\rd x=\frac{1}{\tau}\int_\Omega (uf(u)+b_0)\,\rd x.
	\end{equation*}
Thus, thanks to \eqref{u1a}, we get
\begin{equation*}
	\|h(t)\|_1=\frac{1}{\tau}\int_0^te^{\frac{1}{\tau}(s-t)}\int_\Omega (uf(u)+b_0)\,\rd x\rd s
	\leq C(\|u^{in}\|_1,\tau).
\end{equation*}
	
 On the other hand,  multiplying \eqref{heqn} by $\mathcal{A}h=h-\Delta h$, we obtain
 \begin{equation*}
 	\frac{\tau}{2}\frac{\rd}{\rd t}\int_\Omega(|\nabla h|^2+h^2)\,\rd x+\|h-\Delta h\|^2=\int_\Omega\mathcal{A}^{-1}[uf(u)+b_0](h-\Delta h)\,\rd x=\int_\Omega \left(uf(u)+b_0\right)h \,\rd x.
 \end{equation*}
 Observing that by Young's inequality,
 \begin{align*}
 	\int_\Omega \left(uf(u)+b_0\right)h \,\rd x	=&\int_\Omega\mathcal{A}^{-1}[uf(u)+b_0](h-\Delta h)\,\rd x\\
 	\leq&\frac14\|h-\Delta h\|^2+\|\mathcal{A}^{-1}[uf(u)+b_0]\|^2,
 \end{align*}
 it then follows  that
 \begin{equation}\label{h2a}
 	\frac{\tau}{2}\frac{\rd}{\rd t}\int_\Omega(|\nabla h|^2+h^2)\,\rd x+\int_\Omega \left(uf(u)+b_0\right)h \,\rd x+\frac12\|h-\Delta h\|^2\leq 	2\|\mathcal{A}^{-1}[uf(u)+b_0]\|^2.
 \end{equation}
 Now we recall that by Lemma \ref{lm2}, for $N=2,3$, it holds
 \begin{align*}
 	\|\mathcal{A}^{-1}[uf(u)+b_0]\|^2\leq& C\left(\int_\Omega\left(uf(u)+b_0\right)\,\rd x\right)^2,
 \end{align*}
 and thanks to \eqref{f2a}
 \begin{equation*}
 	\int_\Omega\left(uf(u)+b_0\right)\,\rd x\leq \int_\Omega (a_2u\log u+b_2+b_0)\,\rd x.
 \end{equation*}
 Thus, since $s\log s\geq -e^{-1}$ for $s\geq0$, there is $M>0$ depending only on $f$ and $\Omega$ such that
 \begin{align*}
 2\|\mathcal{A}^{-1}[uf(u)+b_0]\|^2\leq& C\left(\int_\Omega\left(uf(u)+b_0\right)\,\rd x\right)\left(\int_\Omega (a_2u\log u+b_2+b_0)\,\rd x\right)\\
 \leq& M \left(\int_\Omega\left(uf(u)+b_0\right)\,\rd x\right)\left(\int_\Omega (u\log u+e^{-1}+1)\,\rd x\right).
 \end{align*}

 Finally, we deduce  from the compact embedding $H^2\hookrightarrow H^1$ that for any $\lambda_1,\varepsilon>0$
 \begin{align}\label{h2c}
 	\lambda_1\int_\Omega(|\nabla h|^2+h^2)\,\rd x=\lambda_1\|h\|_{H^1}^2
 	\leq \varepsilon\lambda_1\|h-\Delta h\|^2+\lambda_1 C(\varepsilon)\|h\|_1^2.
 \end{align}Then, \eqref{h2} follows by taking $\lambda_1\varepsilon=1/4$ and adding \eqref{h2c} up with \eqref{h2a}.
	This completes the proof.
\end{proof}
Now we  prove the following key comparison lemma which provides  an upper control of $v$ by the two auxiliary functions.
\begin{lemma}\label{com0} There exists a positive constant $C>0$ depending only on the initial data such that for all $(t,x) \in  [0,T_{\mathrm{max}})\times\bar{\Omega}$, it holds that
\begin{equation}\label{comB}
	v(t,x)\leq w(t,x)+\tau h(t,x)+C.
\end{equation}
\end{lemma}
\begin{proof} By \eqref{ks2}, \eqref{defw}, \eqref{heqn} and \eqref{keyid}, one infers that	
	\begin{align*}
		\tau v_t-\Delta v+v=&\,u=w-\Delta w\\
		=&\bigg(\tau w_t-\Delta w+w\bigg)-\tau w_t\\
		=&\bigg(\tau w_t-\Delta w+w\bigg)+\tau ue^{-v}+\tau \mathcal{A}^{-1} [uf(u)]-\tau\mathcal{A}^{-1}[ue^{-v}]\\
		\leq&\bigg(\tau w_t-\Delta w+w\bigg)+\tau ue^{-v}+\tau \mathcal{A}^{-1} [uf(u)+b_0]-\tau\mathcal{A}^{-1}[ue^{-v}]\\
		=&\bigg(\tau w_t-\Delta w+w\bigg)+\tau ue^{-v}+\tau \bigg(\tau h_t-\Delta h+h\bigg)-\tau\mathcal{A}^{-1}[ue^{-v}]\\
		\leq&\bigg(\tau w_t-\Delta w+w\bigg)+\tau ue^{-v}+\tau \bigg(\tau h_t-\Delta h+h\bigg),
	\end{align*}where the last inequality holds since $\tau\mathcal{A}^{-1}[ue^{-v}]\geq0$ due to maximum principles.
On the other hand,  we also notice by \eqref{ks2} that
\begin{align*}
	ue^{-v}=&e^{-v}(\tau v_t-\Delta v+v)\\
	=&-\bigg(\tau \partial_te^{-v}-\Delta e^{-v}+e^{-v}\bigg)-e^{-v}|\nabla v|^2+ve^{-v}+e^{-v}\\
	\leq&-\bigg(\tau \partial_te^{-v}-\Delta e^{-v}+e^{-v}\bigg)+1,
\end{align*}	since $e^{-v}|\nabla v|^2\geq0$ and $ve^{-v}+e^{-v}\leq\max_{s\geq0}(1+s)e^{-s}=1$.

Collecting the above inequalities, we obtain that for $(t,x)\in(0,T_{\mathrm{max}})\times \Omega$,
	\begin{align*}
	\tau v_t-\Delta v+v	
		\leq&\tau\partial_t(w-\tau e^{-v}+\tau h)-\Delta(w-\tau e^{-v}+\tau h)+(w-\tau e^{-v}+\tau h)+\tau\\
		\leq&\tau\partial_t(w-\tau e^{-v}+\tau h+C)-\Delta(w-\tau e^{-v}+\tau h+C)+(w-\tau e^{-v}+\tau h+C)\,,
	\end{align*}
where the constant $C\geq\tau$ is chosen such that for all $x\in\Omega,$
	\begin{equation*}
		\tau e^{-v^{in}}+v^{in}\leq \tau+v^{in}\leq w^{in}+C=\mathcal{A}^{-1}[u^{in}]+C.
	\end{equation*} 
Since $\nabla v\cdot\mathbf{n}=\nabla (w-\tau e^{-v}+\tau h+C)\cdot\mathbf{n}=0$ on $\partial\Omega$, we finally deduce by comparison principle of heat equations and the non-negativity of $e^{-v}$ that  for all $(t,x) \in  [0,T_{\mathrm{max}})\times \bar{\Omega},$
\begin{equation*}
	v(t,x)\leq w(t,x)-\tau e^{-v(t,x)}+\tau h(t,x)+C\leq w(t,x)+\tau h(t,x)+C.
\end{equation*}
This completes the proof.
\end{proof}
With the above comparison result and Lemma \ref{wbound} at hand, it suffices to derive an upper bound for the second auxiliary function $h$ in order to get the $L^\infty$-norm of $v$. To this aim, we need to establish the following crucial uniform $L\log L$-estimate for $u$.
\begin{lemma}\label{lmenergy}Assume that $N\leq 3$.
	 There exists a time-independent constant $C>0$ such that for all $t\in[0,T_{\mathrm{max}})$,
	\begin{equation}
		\int_\Omega (u\log u+e^{-1})\,\rd x+\|v\|_{H^1}^2\leq C.
	\end{equation}
\end{lemma}
\begin{proof}
First, we deduce from \eqref{ks} via direct calculations that
\begin{align}\label{ener1}
	\frac{\rd}{\rd t}\int_\Omega u(\log u-v)\,\rd x=&\int_\Omega u_t(\log u-v)\,\rd x+\int_\Omega u_t\,\rd x-\int_\Omega uv_t \,\rd x\nonumber\\
	=&\int_\Omega \Delta(ue^{-v})(\log u-v)\,\rd x-\int_\Omega uf(u)\log u \,\rd x+ \int_\Omega uf(u)v\,\rd x\nonumber\\
	&-\int_\Omega uf(u)\,\rd x-\int_\Omega(\tau v_t-\Delta v+v)v_t\,\rd x\nonumber\\
	=&-\int_\Omega u e^{-v}|\nabla\log u-\nabla v|^2\,\rd x-\int_\Omega uf(u)\log u \,\rd x+ \int_\Omega uf(u)v\,\rd x\nonumber\\
	&-\int_\Omega uf(u)\,\rd x-\tau\|v_t\|^2-\frac{1}{2}\frac{\rd}{\rd t}\int_\Omega(|\nabla v|^2+v^2)\,\rd x.
\end{align}
Recalling that $uf(u)+b_0\geq0$, we deduce by \eqref{comA} and \eqref{comB} that 
\begin{align*}
	\int_\Omega uf(u)v\,\rd x=&\int_\Omega \left(uf(u)+b_0\right)v\,\rd x-b_0\|v\|_1\\
	\leq&\int_\Omega \left(uf(u)+b_0\right)(w+\tau h+C)\,\rd x\\
	\leq &\int_\Omega \left(uf(u)+b_0\right)w\,\rd x+\tau \int_\Omega \left(uf(u)+b_0\right)h \,\rd x+C\int_\Omega \left(uf(u)+b_0\right)\,\rd x\\
	\leq &\tau \int_\Omega \left(uf(u)+b_0\right)h \,\rd x+C\int_\Omega \left(uf(u)+b_0\right)\,\rd x.
\end{align*}
Observe that by \eqref{f1} and the fact $f\in C([0,\infty)$,  for any $\lambda>0$, there is $C(\lambda)\geq0$ depending on $f$ such that $sf(s)\log s\geq 2\lambda s\log s-C(\lambda)$ for all $s\geq0$. Indeed,  for any $\lambda>0$, there is $s_0(\lambda)\geq1$ such that $f(s)\geq 2\lambda$ for all $s\geq s_0$ due to \eqref{f1}. Since $s\log s\geq0$ when $s\geq s_0\geq1$, it follows that $sf(s)\log s\geq 2\lambda s\log s$ for all $s\geq s_0$. Then, we can pick $C(\lambda)=-\min\{0,\min_{s\in[0,s_0]}\{f(s)s\log s-2\lambda s\log s\}\}$.
Thus, we infer that
\begin{align*}
	\int_\Omega uf(u)\log u \,\rd x\geq 2\lambda&\int_\Omega u\log u\,\rd x-C(\lambda)|\Omega|.
\end{align*}
As a result, we obtain from \eqref{ener1} that
\begin{align}\label{ener3}
	\frac{\rd}{\rd t}\int_\Omega\bigg(u\log u+e^{-1}+1+\frac12|\nabla v|^2&+\frac{1}2v^2-uv\bigg)\,\rd x+2\lambda\int_\Omega (u\log u+e^{-1}+1) \,\rd x\nonumber\\
	\leq& \tau \int_\Omega \left(uf(u)+b_0\right)h \,\rd x+C\int_\Omega \left(uf(u)+b_0\right)\,\rd x+C(\lambda).
\end{align}

In view of \eqref{vl1}, \eqref{comA} and \eqref{dual0},
\begin{equation*}
	\int_\Omega uv \,\rd x=\int_\Omega (w-\Delta w)v\,\rd x=\int_\Omega wv\,\rd x+\int_\Omega \nabla w\cdot\nabla v\,\rd x\leq\|w\|_\infty\|v\|_1+\|\nabla w\|^2+\frac14\|\nabla v\|^2
	\leq C_0+\frac14\|\nabla v\|^2\,,
\end{equation*}
and thus
\begin{align}
	\mathcal{G}(t):=&\int_\Omega\bigg(u\log u+e^{-1}+1+\frac12|\nabla v|^2+\frac{1}2 v^2-uv\bigg) \,\rd x+C_0\nonumber\\
	\geq& \int_\Omega \bigg(u\log u+e^{-1}+1+\frac14|\nabla v|^2+\frac{1}2 v^2\bigg)\,\rd x\geq0.
\end{align}

Define
\begin{align}
	\Lambda(t):=&\mathcal{G}(t)+\lambda\tau\|v\|^2+\frac{\tau^2}{2}(\|\nabla h\|^2+\|h\|^2)\nonumber\\
	=&\int_\Omega\bigg(u\log u+e^{-1}+1+\frac12|\nabla v|^2+\frac{1+2\tau\lambda}{2}v^2-uv+\frac
	{\tau^2}{2}(|\nabla h|^2+h^2)\bigg) \,\rd x+C_0.
\end{align}
Invoking \eqref{ener3}, \eqref{ener2}, and \eqref{h2} by taking $\lambda_1=\tau\lambda$, we derive   that
\begin{align*}
	\frac{\rd}{\rd t}\Lambda(t)=&\frac{\rd}{\rd t}\mathcal{G}(t)+\tau\lambda\frac{\rd}{\rd t}\int_\Omega v^2\,\rd x+\frac{\tau^2}{2}\frac{\rd}{\rd t}\int_\Omega(|\nabla h|^2+h^2)\,\rd x\\
	\leq&\tau \int_\Omega \left(uf(u)+b_0\right)h \,\rd x+C\int_\Omega \left(uf(u)+b_0\right)\,\rd x-2\lambda\int_\Omega (u\log u+e^{-1}+1)\,\rd x\\
	&+C\lambda-\lambda\int_\Omega (|\nabla v|^2+v^2)\,\rd x\\
	&+M\tau\left(\int_\Omega(u\log u+e^{-1}+1)\,\rd x\right)\left(\int_\Omega (uf(u)+b_0)\,\rd x\right)+C(\tau,\lambda)\\
	&-\tau^2\lambda\int_\Omega(|\nabla h|^2+h^2)\,\rd x-\tau\int_\Omega \left(uf(u)+b_0\right)h \,\rd x-\frac\tau4\|h-\Delta h\|^2,
\end{align*}
which leads to
\begin{align*}
	\frac{\rd}{\rd t}&\Lambda(t)+2\lambda\Lambda(t)\\
	=&\frac{\rd}{\rd t}\Lambda(t)+2\lambda\int_\Omega (u\log u +e^{-1}+1)\,\rd x+\lambda\int_\Omega|\nabla v|^2\,\rd x-2\lambda\int_\Omega uv \,\rd x\\
	&+\lambda\tau^2\int_\Omega (|\nabla h|^2+h^2)\,\rd x+\lambda(1+2\tau\lambda)\|v\|^2\\
\leq&M\tau\left(\int_\Omega(u\log u+e^{-1}+1)\,\rd x\right)\left(\int_\Omega (uf(u)+b_0)\,\rd x\right)+	C\int_\Omega \left(uf(u)+b_0\right)\,\rd x+C(\tau,\lambda)\\
\leq& M\tau\Lambda(t)\left(\int_\Omega (uf(u)+b_0)\,\rd x\right)+C\left(\int_\Omega (uf(u)+b_0)\,\rd x\right)+C(\tau,\lambda),
\end{align*}where we use the facts that $\sup_{t\in[0,T_{\mathrm{max}})}\|v\|\leq C$, and $\int_\Omega \left(u\log u+e^{-1}+1\right)\rd x\leq \mathcal{G}(t)\leq \Lambda(t).$

Now, we recall that by  \eqref{u1b} and take $\lambda=2M\tau b_0$,
\begin{equation}\label{u1b0}
	M\tau\int_s^t\int_\Omega  (uf(u)+b_0)\,\rd x\rd \sigma\leq 2M\tau b_0(t-s)+M\tau\big(2\|u^{in}\|_1+4b_0|\Omega|\big)=\lambda(t-s)+2M\tau\|u^{in}\|_1+4b_0M\tau|\Omega|\,,
\end{equation}and according to \eqref{u1a} with $\varepsilon=\lambda$,
\begin{equation*}
\int_0^te^{\lambda(s-t)}\int_\Omega (uf(u)+b_0)\,\rd x\rd s\leq C(\lambda).
\end{equation*}
Finally, thanks to Lemma \ref{lemgron}, we conclude that
\begin{equation*}
\sup_{t\in[0,T_{\mathrm{max}})}	\Lambda(t)\leq C
\end{equation*}with $C>0$ independent of time. This completes the proof.
\end{proof}
\begin{lemma}\label{hvbound}
	Assume that $N\leq 3$. Then, there is a time-independent constant $C>0$ such that for all $t\in[0,T_{\mathrm{max}})$,
	\begin{equation}
		\|h\|_\infty+\|v\|_\infty\leq C.
	\end{equation}
\end{lemma}
\begin{proof}
	According to \eqref{f1b} and Lemma \ref{lmenergy},
	\begin{equation}\label{fL1}
		\int_\Omega (uf(u)+b_0)\,\rd x\leq \int_\Omega \left(a_2(u\log u+e^{-1})+b_2\right)\,\rd x\leq C,
	\end{equation} we infer from Lemma \ref{lm2} that for $N\leq 3$ and all $t\in[0,T_{\mathrm{max}})$
	\begin{equation}
		\|\mathcal{A}^{-1}[uf(u)+b_0]\|_q\leq C(q)\int_\Omega (uf(u)+b_0)\,\rd x\leq C
	\end{equation}with some $q>N/2$. Recall that by theory of semi-groups,
	\begin{equation*}
		h(t,x)=\int_0^te^{\frac{1}{\tau}\mathcal{A}(t-s)}\mathcal{A}^{-1}[uf(u)+b_0](s)\,\rd s,
	\end{equation*}
which together standard $L^p-L^q$ estimates leads us to
	\begin{align*}
		\|h\|_\infty\leq \int_0^te^{\frac{1}{\tau}(s-t)}(t-s)^{-\frac{N}{2q}}	\|\mathcal{A}^{-1}[uf(u)+b_0](s)\|_q\,\rd s\leq C\,,
	\end{align*}since for any $\theta\in(0,1)$,
	\begin{equation*}
		\mathcal{I}_\theta:=\int_0^\infty s^{-\theta}e^{-s/\tau}\,\rd s<\infty\,.
	\end{equation*}
	This completes the proof in view of  Lemma \ref{com0} and Lemma \ref{wbound}.
\end{proof}

\section{Uniform-in-time boundedness of $u$ in 2D.}
In this section, we assume $N=2$ and we derive the $L^\infty$-norm of $u$, which entails the existence of classical solutions as well as their uniform-in-time boundedness.

	Since now $\sup_{t\in[0,T_{\mathrm{max}})}(\|v\|_\infty+\|w\|_{H^1})\leq C$, an integration of \eqref{dual1} implies that
\begin{lemma}\label{cor1}
For any $0<s<t<T_{\mathrm{max}}$, there are $M_1, M_2>0$ depending only on $f$, $\Omega$ and the initial data such that
	\begin{equation}\label{dual3}
		\int_s^t\int_\Omega u^2\,\rd x\rd \sigma\leq M_1(t-s)+M_2\,.
	\end{equation}
\end{lemma}

\begin{lemma}\label{leml2u}
	Assume $N=2$. There is $C>0$ independent of time such that for all $t\in[0,T_{\mathrm{max}})$,
	\begin{equation*}
		\|u(t)\|+\|\nabla v\|_4\leq C.
	\end{equation*}
\end{lemma}
\begin{proof} The proof can be carried out in a similar way as in \cite[Lemma 3.3]{BBTW2015} using an extended interpolation inequality attributed to Biler et al. \cite{Biler1994}. Here, we choose an alternative manner relying on Lemma \ref{cor1}.

First,	using that $\nabla v\cdot\nabla \Delta v=\frac12\Delta|\nabla v|^2-|\nabla^2 v|^2$, we see from \eqref{ks2} that
\begin{align*}
	\frac\tau4\int_\Omega |\nabla v|^4\,\rd x=&\int_\Omega|\nabla v|^2\nabla v\cdot\nabla (\Delta v-v+u)\,\rd x\\
	=&-\frac12\int_\Omega |\nabla |\nabla v|^2|^2\,\rd x+\frac12\int_{\partial\Omega} |\nabla v|^2\frac{\partial|\nabla v|^2}{\partial\mathbf{n}}\rd s-\int_\Omega|\nabla v|^2|\nabla^2 v|^2\,\rd x\\
	&-\int_\Omega |\nabla v|^4\,\rd x-\int_\Omega u|\nabla v|^2\Delta v\,\rd x-\int_\Omega u\nabla v\cdot\nabla |\nabla v|^2\,\rd x\,.
\end{align*}
Here since $\frac{\partial|\nabla v|^2}{\partial\mathbf{n}}\leq \kappa_1|\nabla v|^2$ on $\partial\Omega$ with some $\kappa_1>0$ depending on $\Omega$ \cite{MizoSoup2014}, by the compactness of the embedding $H^1(\Omega)\hookrightarrow L^2(\partial\Omega)$, we deduce that
\begin{equation*}
\frac12\int_{\partial\Omega} |\nabla v|^2\frac{\partial|\nabla v|^2}{\partial\mathbf{n}}\rd s\leq \frac14\int_\Omega |\nabla|\nabla v|^2|^2 \,\rd x+C\bigg(\int_\Omega |\nabla v|^2\,\rd x\bigg)^2.
\end{equation*}
Moreover, by Young's inequality
\begin{align*}
	-\int_\Omega u|\nabla v|^2\Delta v\,\rd x-\int_\Omega u\nabla v\cdot\nabla |\nabla v|^2\,\rd x\leq \frac18\int_\Omega |\nabla|\nabla v|^2|^2\,\rd x+\int_\Omega |\nabla v|^2|\nabla^2 v|^2\,\rd x+C\int_\Omega u^2|\nabla v|^2\,.
\end{align*}
On the other hand, from \eqref{ks1}, we infer that
	\begin{align*}
		\frac{\rd}{\rd t}\int_\Omega u^2\,\rd x=&-2\int_\Omega e^{-v}|\nabla u|^2\,\rd x-2\int_\Omega u^2f(u)\,\rd x-2\int_\Omega e^{-v}u\nabla u\cdot\nabla v\,\rd x\\
		\leq&-\int_\Omega  e^{-v}|\nabla u|^2\,\rd x-2\int_\Omega u^2f(u)\,\rd x+\int_\Omega e^{-v}u^2|\nabla v|^2\,\rd x\\
		\leq &-\int_\Omega  e^{-v}|\nabla u|^2\,\rd x-2\int_\Omega u^2f(u)\,\rd x+\int_\Omega u^2|\nabla v|^2\,\rd x\,,
	\end{align*}which together with the fact $\sup_{t\in[0,T_{\mathrm{max}})}(\|u\|_1+\|v\|_\infty)\leq C$ entails that there is $\kappa_0>0$ such that
\begin{equation}\label{ul2a}
	\frac{\rd}{\rd t}\int_\Omega u^2\,\rd x+2\kappa_0\int_\Omega |\nabla u|^2\,\rd x+2\int_\Omega u(uf(u)+b_0)\,\rd x\leq \int_\Omega u^2|\nabla v|^2 \,\rd x+C\,.
\end{equation}	
	
Observe that by the Gagliardo-Nirenberg inequality,
\begin{equation*}
	\|u\|_3^3\leq C\|\nabla u\|\|u\|^2+C\|u\|^3,
\end{equation*}and 
\begin{equation*}
	\||\nabla v|^2\|^3_3\leq C\|\nabla|\nabla v|^2\|^2\||\nabla v|^2\|_1+C\||\nabla v|^2\|_1^3.
\end{equation*}
Recall that $\sup_{t\in[0,T_{\mathrm{max}})}\|\nabla v\|\leq C$. We deduce that from above and Young's inequality that
\begin{align*}
	\int_\Omega u^2|\nabla v|^2\,\rd x\leq& \|u\|_3^2\||\nabla v|^2\|_3\\
	\leq&C\|u\|_3^2\bigg(\int_\Omega|\nabla|\nabla v|^2|^2\,\rd x\bigg)^{\frac13}+C\|u\|_3^2\\
	\leq &\varepsilon\int_\Omega \int_\Omega|\nabla|\nabla v|^2|^2\,\rd x+C_\varepsilon\|u\|_3^3+C\\
	\leq & \varepsilon\int_\Omega \int_\Omega|\nabla|\nabla v|^2|^2\,\rd x+C_\varepsilon\|\nabla u\|\|u\|^2+C_\varepsilon\|u\|^3+C\\
	\leq&\varepsilon\int_\Omega \int_\Omega|\nabla|\nabla v|^2|^2\,\rd x+\varepsilon\|\nabla u\|^2+C_{\varepsilon}\|u\|^4+C_{\varepsilon}
\end{align*}with any $\varepsilon>0.$

Now, collecting the above estimates and taking $\varepsilon$ sufficiently small, we arrive at
\begin{equation*}
	\frac{\rd}{\rd t}\int_\Omega (u^2+\frac{\tau}{4}|\nabla v|^4)\,\rd x+\kappa_0\|\nabla u\|^2+\frac18\int_\Omega|\nabla|\nabla v|^2|^2\,\rd x\leq C_1\|u\|^4+C.
\end{equation*}
Since $H^1(\Omega)\hookrightarrow L^2(\Omega)$ is compact, we deduce that with any $\lambda_2>0$, 
\begin{equation*}
	\frac{\rd}{\rd t}\int_\Omega (u^2+\frac{\tau}{4}|\nabla v|^4)\,\rd x+2\lambda_2\int_\Omega (u^2+\frac{\tau}{4}|\nabla v|^4)\,\rd x\leq C_1\|u\|^2\int_\Omega (u^2+\frac{\tau}{4}|\nabla v|^4)\,\rd x+C(\lambda_2).
\end{equation*}
Finally,  taking $\lambda_2=C_1M_1$, with $M_1$ defined in \eqref{dual3}, and applying Lemma \ref{lemgron}, we conclude that 
\begin{equation*}
	\sup_{t\in[0,T_{\mathrm{max}})}(\|u\|+\|\nabla v\|_4)\leq C
\end{equation*}with $C>0$ independent of time. This completes the proof.
\end{proof}
\begin{remark}
Thanks to  the following	two-dimensional Gagliardo-Nirenberg inequalities:
\begin{equation*}
	\|u\|_4^2\leq C\|\nabla u\|\|u\|+C\|u\|^2\,,
\end{equation*}and
\begin{equation}
	\|\nabla v\|_4^2\leq C\|v\|_{H^2}\|v\|_\infty+C\|v\|_\infty^2\,,
\end{equation}
when $\tau=0$, we notice that $\|v\|_{H^2}\leq C\|u\|$ and hence
	\begin{align*}
		\int_\Omega u^2|\nabla v|^2\,\rd x\leq& \|u\|_4^2\|\nabla v\|_4^2\\
		\leq&C\big(\|\nabla u\|\|u\|+\|u\|^2\big)\big(\|v\|_{H^2}\|v\|_\infty+\|v\|_\infty\big)\\
		\leq&C\big(\|\nabla u\|\|u\|+\|u\|^2\big)(\|u\|+1)\\
		\leq&C\big(\|\nabla u\|\|u\|^2+\|\nabla u\|\|u\|+\|u\|^3+\|u\|^2\big)\\
		\leq&\frac{\kappa_0}{2}\|\nabla u\|^2+C(\kappa_0)(\|u\|^4+1).
	\end{align*}
Then, we may derive the uniform $L^2$-norm from \eqref{ul2a} in the same manner.	
\end{remark}

\textbf{Proof of Theorem \ref{Th1} in 2D.} Based  on Lemma \ref{leml2u}, we can further prove by Moser-Alikakos iteration that $\sup_{t\in[0,T_{\mathrm{max}})}\|u\|_\infty<C$ with some $C>0$ independent of time (see, e.g., \cite[Lemma 12]{FuJi2021a}, \cite[Lemma A.1]{TW2012}), which together with Proposition \ref{prop2.1} proves the global existence of uniform bounded classical solutions to problem \eqref{ks} in the two-dimensional setting.


\section{Boundedness of classical solutions in 3D}
In this section, we consider the case $N=3$. Uniform bounds of $u$ will be established by an application of semi-group theory. We remark that at the current stage, pure energy method seems inadequate to achieve the same goal with a minimal requirement on $\gamma$ or $f$ in higher dimensions \cite{FuJi2021b,LyWa2023}.

\subsection{H\"older continuity of $v$.}
First, we show the H\"older continuity of $w$. For simplicity, we denote $\varphi=\mathcal{A}^{-1}[ue^{-v}]$ and $\psi=\mathcal{A}^{-1}[uf(u)+b_0]$. We observe from  the key identity \eqref{keyid} and \eqref{defw} that $w$ solves the initial-boundary value problem
\begin{subequations}\label{cpv}
	\begin{align}
		&\partial_t w + ue^{-v}  = \varphi-\psi+b_0\,, \quad &(t,x)&\in (0,T_{\mathrm{max}})\times\Omega\,, \label{cpv1} \\
		&- \Delta w +  w  = u \,, \quad &(t,x)&\in (0,T_{\mathrm{max}})\times\Omega\,, \label{cpv1.5} \\
		&\nabla w \cdot \mathbf{n}  = 0\,, \quad &(t,x)&\in (0,T_{\mathrm{max}})\times\partial\Omega\,, \label{cpv2} \\
		&w(0)   = w^{in}\,, \quad &x&\in\Omega\,. \label{cpv3}
	\end{align}
\end{subequations}
We now fix $T\in(0,T_{\mathrm{max}})$ and set $J_T \triangleq [0,T]$. By Lemma~\ref{wbound}, and Lemma~\ref{hvbound}, there are positive constants $v^*$, $w^*$ and $\gamma_*$ such that
\begin{align}
	0  \le v(t,x) & \le v^*\;\;\text{ in }\;\; J_T\times \bar{\Omega}\,, \label{bv}\\
	0  \le w(t,x) & \le w^*\;\;\text{ in }\;\; J_T\times \bar{\Omega}\,, \label{bw}\\
	0<\gamma_*=e^{-v^*}&\leq e^{-v}\leq 1\;\;\text{ in }\;\; J_T\times \bar{\Omega}\,\label{bg},
\end{align}
and we infer from \eqref{defw}, \eqref{bw}, and the elliptic comparison principle that
\begin{equation}
0\leq\varphi=\mathcal{A}^{-1}[ue^{-v}]\leq\mathcal{A}^{-1}[u]=w\leq w^*.\label{phib}
\end{equation}
Besides, we recall that since $\sup_{t\in[0,T]}\|uf(u)+b_0\|_1\leq C$ due to \eqref{fL1}, by Lemma \ref{lm2} with $q=2$, it holds that
\begin{equation*}
	\sup_{t\in[0,T]}\|\psi\|\leq C\|uf(u)+b_0\|_1\leq C.
\end{equation*}
After the above preparation, we now proceed along the lines of \cite[Chapter~V, Section~7]{LSU1968} and \cite[Section~5]{Aman1989} to derive a local energy bound for $w$. Note that the same local energy estimate holds for $v$ when $\tau=0$.

\begin{lemma}\label{lem.impreg1}  Let $\delta\in (0,1)$. There is $C>0$ independent of $T$ such that, if $\vartheta\in C^\infty(J_T\times\bar{\Omega})$, $0\le \vartheta\le 1$, $\sigma\in\{-1,1\}$, and $\rho\in\mathbb{R}$ are such that
	\begin{equation}
		\sigma w(t,x) - \rho \le \delta\,, \qquad (t,x)\in \mathrm{supp}\,\vartheta\,, \label{smallc1}
	\end{equation}
	then
	\begin{align*}
		& \int_\Omega \vartheta^2 (\sigma w(t) - \rho)_+^2\ \mathrm{d}x + \frac{\gamma_*}{2} \int_{t_0}^t \int_\Omega \vartheta^2 |\nabla (\sigma w(s) -\rho)_+|^2\ \mathrm{d}x\mathrm{d}s \\
		& \qquad  \le \int_\Omega \vartheta^2 (\sigma w(t_0) - \rho)_+^2\ \mathrm{d}x + C \int_{t_0}^t \int_\Omega  \left( |\nabla\vartheta|^2 + \vartheta|\partial_t\vartheta| \right) (\sigma w(s) - \rho)_+^2\ \mathrm{d}x\mathrm{d}s \\
		& \qquad\quad + C \int_{t_0}^t  \bigg(\int_{A_{\rho,\sigma}(s)} \vartheta\ \mathrm{d}x\bigg)^{1/2} \mathrm{d}s 
	\end{align*}
	for $0\le t_0\le t\le T$, where 
	\begin{equation*}
		A_{\rho,\sigma}(s)\triangleq \left\{ x\in \Omega\ :\ \sigma w(s,x) > \rho \right\}\,, \qquad s\in [0,T]\,.
	\end{equation*}
\end{lemma}

\begin{proof}
	By \eqref{cpv},
	\begin{align}
		\frac{1}{2} \frac{\rd}{\rd t} \int_\Omega \vartheta^2 ( \sigma w - \rho)_+^2\ \mathrm{d}x &  =  \sigma \int_\Omega \vartheta^2 (\sigma w - \rho)_+ \partial_t w \ \mathrm{d}x + \int_\Omega  (\sigma w - \rho)_+^2 \vartheta \partial_t\vartheta\ \mathrm{d}x \nonumber \\
		& = - \sigma \int_\Omega \vartheta^2 (\sigma w - \rho)_+ u e^{-v} \,\rd x + \sigma \int_\Omega \vartheta^2 (\sigma w - \rho)_+ (\varphi-\psi+b_0)\ \mathrm{d}x \label{Z1}\\
		& \qquad + \int_\Omega (\sigma w - \rho)_+^2 \vartheta \partial_t\vartheta \mathrm{d} x\,. \nonumber
	\end{align}
	Either $\sigma=1$ and it follows from  \eqref{cpv1.5}, \eqref{bg}, and the non-negativity of $u$ and $w$ that
	\begin{align*}
		- \sigma \int_\Omega \vartheta^2 (\sigma w - \rho)_+ u e^{-v} \,\rd x & \le - \gamma_*\int_\Omega \vartheta^2 (w - \rho)_+ u\ \mathrm{d}x \\
		& = - \gamma_* \int_\Omega \vartheta^2 (w - \rho)_+ ( w - \Delta w)\ \mathrm{d}x \\
		& \le - \gamma_* \int_\Omega \nabla \left[ \vartheta^2 (w - \rho)_+ \right]\cdot \nabla w\ \mathrm{d}x \\
		& \le - \gamma_* \int_\Omega \vartheta^2 |\nabla (w - \rho)_+ |^2\ \mathrm{d}x \\
		& \qquad + 2 \gamma_* \int_\Omega \vartheta |\nabla\vartheta| (w-\rho)_+ |\nabla (w-\rho)_+|\ \mathrm{d}x\,.
	\end{align*}
	Or $\sigma=-1$, we infer from  \eqref{cpv1.5}, \eqref{bg} and \eqref{bw} that
	\begin{align*}
		- \sigma \int_\Omega \vartheta^2 (\sigma w - \rho)_+ u e^{-v} \,\rd x & \le   \int_\Omega \vartheta^2 (-w - \rho)_+ u\ \mathrm{d}x \\
		& =  \int_\Omega \vartheta^2 (-w - \rho)_+ ( w - \Delta w)\ \mathrm{d}x \\
		& \le  w^* \int_\Omega \vartheta^2 (-w - \rho)_+\ \mathrm{d}x \\	
		& \qquad + \int_\Omega \nabla \left[ \vartheta^2 (-w - \rho)_+ \right]\cdot \nabla w\ \mathrm{d}x \\
		& \le -  \int_\Omega \vartheta^2 |\nabla(-w-\rho)_+|^2\ \mathrm{d}x \\
		& \qquad + 2  \int_\Omega \vartheta |\nabla\vartheta| (-w-\rho)_+ |\nabla (-w-\rho)_+|\ \mathrm{d}x \\
		& \qquad + w^* \int_\Omega \vartheta^2 (-w - \rho)_+ \ \mathrm{d}x\,.
	\end{align*}
	Since $\gamma_*\leq 1$, we have thus shown  that
	\begin{align*}
		- \sigma \int_\Omega \vartheta^2 (\sigma w - \rho)_+ u e^{-v} \,\rd x & \le - \gamma_* \int_\Omega \vartheta^2 |\nabla(\sigma w-\rho)_+|^2\ \mathrm{d}x \\
		& \qquad + 2 \int_\Omega \vartheta |\nabla\vartheta| (\sigma w-\rho)_+ |\nabla (\sigma w-\rho)_+|\ \mathrm{d}x \\
		& \qquad + w^* \int_\Omega \vartheta^2 (\sigma w - \rho)_+\ \mathrm{d}x\,.
	\end{align*}
	Inserting this estimate in \eqref{Z1} and using \eqref{phib} and Young's inequality lead us to
	\begin{align*}
		\frac{1}{2} \frac{\rd}{\rd t} \int_\Omega \vartheta^2 ( \sigma w - \rho)_+^2\ \mathrm{d}x & \le - \gamma_* \int_\Omega \vartheta^2 |\nabla(\sigma w-\rho)_+|^2\ \mathrm{d}x \\
		& \qquad + \frac{\gamma_*}{2} \int_\Omega \vartheta^2 |\nabla(\sigma w-\rho)_+|^2\ \mathrm{d}x +  C \int_\Omega |\nabla\vartheta|^2 (\sigma w-\rho)_+^2\ \mathrm{d}x \\
		& \qquad + (2w^*+b_0) \int_\Omega \vartheta^2 (\sigma w - \rho)_+\ \mathrm{d}x +   \int_\Omega \vartheta^2 (\sigma w -\rho)_+|\psi|\ \mathrm{d}x \\
		& \qquad + \int_\Omega (\sigma w - \rho)_+^2 \vartheta |\partial_t\vartheta| \mathrm{d} x \\
		& \le -\frac{\gamma_*}{2} \int_\Omega \vartheta^2 |\nabla(\sigma w-\rho)_+|^2\ \mathrm{d}x \\
		& \qquad +  C \int_\Omega \left( |\nabla\vartheta|^2 + \vartheta |\partial_t\vartheta| \right) (\sigma w-\rho)_+^2\ \mathrm{d}x\\
		&\qquad +\int_\Omega \vartheta^2 (\sigma w - \rho)_+|\psi|\ \mathrm{d}x + (2w^*+b_0) \int_\Omega \vartheta^2 (\sigma w - \rho)_+\ \mathrm{d}x\,.
	\end{align*}
Using \eqref{smallc1} and the fact $0\leq \vartheta\leq 1$, we notice that
	\begin{align*}
		\int_\Omega \vartheta^2 (\sigma w -\rho)_+|\psi|\ \mathrm{d}x \leq&\bigg(\int_\Omega \vartheta^4 |(\sigma w - \rho)_+|^2\ \mathrm{d}x \bigg)^{\frac12}\|\psi\|\\
		\leq& \delta\sup_{t\in[0,T]}\|\psi\|\bigg(\int_{A_{\rho,\sigma}} \vartheta\,\rd x\bigg)^{\frac12}\\
		\leq& C\bigg(\int_{A_{\rho,\sigma}} \vartheta\,\rd x\bigg)^{\frac12}
	\end{align*}
and 
	\begin{equation*}
		 \int_\Omega \vartheta^2 (\sigma w - \rho)_+\ \mathrm{d}x \leq|\Omega|^{\frac12}\bigg(\int_\Omega \vartheta^4 |(\sigma w - \rho)_+|^2\,\rd x\bigg)^{\frac12} \le |\Omega|^{\frac12}\delta \bigg(\int_{A_{\rho,\sigma}} \vartheta\,\rd x\bigg)^{\frac12}\le C\bigg(\int_{A_{\rho,\sigma}} \vartheta\,\rd x\bigg)^{\frac12}\,
	\end{equation*}
	and integrate the above differential inequality over $(t_0,t)$ to complete the proof.
\end{proof}

We are now in a position to apply \cite[Chapter~II, Theorem~8.2]{LSU1968} to obtain a uniform-in-time H\"older estimate for $w$. In this connection, we emphasize that the validity of the local energy estimates derived in Lemma~\ref{lem.impreg1} does not require the test function $\vartheta$ to be compactly supported in $(0,T)\times \Omega$ and thus include also information on the behavior of $w$ on $(0,T)\times \partial\Omega$.

\begin{corollary}\label{regularw}
	There is $\alpha\in (0,1)$ depending on $\Omega$, $f$,  $\tau$,  and the initial data such that $w\in BUC^{\alpha}(J_T,C^{\alpha}(\bar{\Omega}))$.
\end{corollary}

\begin{proof}
	Let $\delta\in (0,1)$. It follows from Lemma~\ref{lem.impreg1} that the estimate \cite[Chapter~II, Equation~(7.5)]{LSU1968} holds true with parameters $r=\frac73, q=\frac{14}{3}, \kappa=\frac16$ satisfying \cite[Chapter~II, equation~(7.3)]{LSU1968}). Consequently, according to \cite[Chapter~II, Remark~7.2 and Remark~8.1]{LSU1968}, there is $C>0$ such that 
	\begin{equation*}
		w\in \hat{\mathcal{B}}_2(J_T\times\bar{\Omega},w^*, C,\frac73,\delta,\frac16)\,,
	\end{equation*} 
	the class of functions $\hat{\mathcal{B}}_2$ being defined in \cite[Chapter~II, Section~8]{LSU1968}. Taking also into account the smoothness of the boundary of $\Omega$ and the H\"older continuity of $w^{in}\in C^{\alpha_0}(\bar{\Omega})$ for some $\alpha_0\in (0,1)$, which stems from the definition $w^{in}=\mA^{-1}[u^{in}]$, the regularity  of $u^{in}$, and elliptic regularity, we then infer from \cite[Chapter~II, Lemma~8.1 \& Theorem~8.2]{LSU1968} that $w\in C^{\alpha}(J_T\times \bar{\Omega})$ for some $\alpha\in (0,1)$, with $\alpha$ depending on other parameters as indicated in the statement of Corollary~\ref{regularw}. Moreover, since the local energy estimate is time-independent, we conclude that the H\"older estimate of $w$ is uniform-in-time as well.
\end{proof}
In the same vein, if $\tau=0$, we obtain a uniform H\"older estimate for $v$. When $\tau>0$, we need to use the standard regularity theory of parabolic equations to obtain a H\"older estimate for $v$ as in \cite[Lemma~3.1]{FuSe2022a}. 

\begin{proposition}\label{prop.impreg2}
	The function $v$ belongs to $BUC^{1+\alpha}(J_T, C^{\alpha}(\bar{\Omega}))$ with the exponent  $\alpha$ given in Corollary~\ref{regularw}.
\end{proposition}

\begin{proof}
	Set $r \triangleq \mA^{-1}[v]$ and  $r^{in} \triangleq\mA^{-1}[v^{in}]$. In view of the regularity $v^{in}\in W^{1,\infty}(\Omega)$, there is $\alpha_0\in(0,1)$ such that $r^{in}\in C^{2+\alpha_0}(\bar{\Omega})$. Besides, we infer from \eqref{ks2} that $r$ is a solution to
	\begin{equation*}
		\begin{split}
			& \tau \partial_t r - \Delta r +  r = w\,, \qquad (t,x)\in (0,T_{\mathrm{max}})\times \Omega\,, \\
			& \nabla r\cdot \mathbf{n} =0\,, \qquad (t,x)\in (0,T_{\mathrm{max}})\times \partial\Omega\,, \\
			& r(0) = r^{in}\,, \qquad x\in \Omega\,,
		\end{split}
	\end{equation*}
	so that standard regularity theory of heat equations, along with Corollary~\ref{regularw}, ensures that $r$ belongs $BUC^{1+\alpha}(J_T, C^{2+\alpha}(\bar{\Omega}))$. As a result, we obtain that $v =  r-\Delta r\in BUC^{1+\alpha}(J_T, C^{\alpha}(\bar{\Omega}))$.
\end{proof}


\subsection{$L^p$-estimates for $u$.}
In this part, we first show that  $|\nabla v|$ is uniformly bounded in $L^r(\Omega)$ with some $r>3$ by semi-group theory. Then we establish $L^p$-boundedness of $u$ by standard energy method with any $p\in(1,\infty)$.
\begin{lemma}\label{lem.impreg3} Let  $T\in(0,T_{\mathrm{max}})$. For any
	$p\in(1,3)$ and $\theta\in \left( \frac{1+p}{2p},1 \right)$, there is $C(p,\theta)>0$ independent of $T$ and $T_{\mathrm{max}}$\,, such that
	\begin{equation*}
		\|w(t)\|_{W^{2\theta,p}} \le C(p,\theta)\,, \qquad t\in [0,T]\,. 
	\end{equation*}
\end{lemma}

\begin{proof} We set $D_T\triangleq  (-1,2v^*)$ and $J_T=[0,T]$, where $v^*$ is defined in  \eqref{bv}. For $s\in D_T$, we define the elliptic operator $\mathsf{A}(s)$ by $\mathsf{A}(s)z := -e^{-s}\Delta z$ and the boundary operator $\mathsf{B}z := \nabla z\cdot \mathbf{n}$. We point out that $\mathsf{A}(s)$ is strongly uniformly elliptic for $s\in D_T$ since  $\min_{D_T}\{\gamma\}>0$. Then, according to \cite[Theorem~4.2]{Aman1990}, $(\mathsf{A}(s),\mathsf{B})$ is normally elliptic for $s\in D_T$. For $t\in J_T$, let $\tilde{\mathsf{A}}(t)$ be the $L^p$-realization of $(\mathsf{A}(v(t)),\mathsf{B})$ with domain 
	\begin{equation*}
		W_{\mathsf{B}}^{2,p}(\Omega) \triangleq \{ z \in W^{2,p}(\Omega)\ :\ \nabla z\cdot \mathbf{n} = 0 \;\;\text{ on}\;\; \partial\Omega \}\,.
	\end{equation*}
	As in the proof of \cite[Theorem~6.1]{Aman1989},  Proposition~\ref{prop.impreg2} guarantee that 
	\begin{equation*}
		\tilde{\mathsf{A}} \in BUC^{\alpha}(J_T, \mathcal{L}(W^{2,p}(\Omega),L^p(\Omega)))
	\end{equation*}
	and that $\tilde{\mathsf{A}}(J_T)$ is a \textit{regularly bounded subset} of $C^{\alpha}(J_T,\mathcal{H}(W^{2,p}(\Omega),L^p(\Omega)))$ in the sense of \cite[Section~4]{Aman1988} (or, equivalently, the condition \cite[(II.4.2.1)]{Aman1995} is satisfied). By \cite[Theorem~A.1]{Aman1990}, see also \cite[Theorem~II.4.4.1]{Aman1995}, there is a unique parabolic fundamental solution $\tilde{U}$ associated to $\{ \tilde{\mathsf{A}}(t)\ :\ t\in J_T \}$ and there exist time-independent positive constants $K>0$ and $\omega>0$ such that 
	\begin{equation}
		\|\tilde{U}(t,s)\|_{\mathcal{L}(W^{2,p}(\Omega))} + 	\|\tilde{U}(t,s)\|_{\mathcal{L}(L^p(\Omega))} + (t-s) 	\|\tilde{U}(t,s)\|_{\mathcal{L}(L^p(\Omega),W^{2,p}(\Omega))} \le K e^{\omega(t-s)} \label{irb2}
	\end{equation}
	for $0\le s < t\leq T$. Since $\theta\in \left( \frac{1+p}{2p},1 \right)$, it follows from \cite[Theorem~5.2]{Aman1993} that
	\begin{equation*}
		\left( L^p(\Omega) , W_{\mathsf{B}}^{2,p}(\Omega) \right)_{\theta,p} \doteq W_{\mathsf{B}}^{2\theta,p}(\Omega) \triangleq \{ z \in W^{2\theta,p}(\Omega)\ :\ \nabla z\cdot \mathbf{n} = 0 \;\;\text{ on}\;\; \partial\Omega \}\,,
	\end{equation*}
	and we infer from \cite[Lemma~II.5.1.3]{Aman1995} that there is time-independent $K_{\theta}>0$ such that
	\begin{equation}
		\|\tilde{U}(t,s)\|_{\mathcal{L}(W_{\mathsf{B}}^{2\theta,p}(\Omega))} + (t-s)^\theta \|\tilde{U}(t,s)\|_{\mathcal{L}(L^p(\Omega),W_{\mathsf{B}}^{2\theta,p}(\Omega))} \le K_{\theta} e^{\omega(t-s)} \label{irb3}
	\end{equation}
	for $0\le s<t\leq T$. We then pick $\mu>\omega$ and deduce from \eqref{cpv} that $w$ solves 
	\begin{equation}
		\begin{split}
			\partial_t w + \big(\mu+\tilde{\mathsf{A}}(\cdot) )w  & = F\,, \qquad t\in J_T\,, \\
			w(0) & = w^{in} \,, 
		\end{split}\label{irb4}
	\end{equation}
	where
	\begin{equation*}
		F(t,x) \triangleq  \big( \varphi - \psi -  we^{-v}+\mu w\big)(t,x)+b_0\,, \qquad (t,x)\in J_T\times \Omega\,. 
	\end{equation*}
	We recall that, by Lemma \ref{lm2}, \eqref{fL1}, \eqref{wbound}, and \eqref{phib}, for any $1<p<3$ \label{cst7}
	\begin{equation}
		\| F(t)\|_p \le \big((2+\mu)w^* +b_0\big)|\Omega|^{\frac1p}+\|\psi\|_p\leq C \,, \qquad t\in J_T\,, \label{irb5}
	\end{equation}
	while the continuity of $u$, $v$ and $w$, see Proposition~\ref{prop2.1}, ensures that
	\begin{equation}
		F+\mu w \in C(J_T\times\bar{\Omega})\,. \label{irb8}
	\end{equation}
	Moreover, $w\in C(J_T,W^{2,p}(\Omega))\cap C^1((0,T],L^p(\Omega))$ satisfies $w(t)\in \mathrm{dom}(\tilde{\mathsf{A}}(t))$ for all $t\in (0,T]$. These properties, along with  \eqref{cpv} and \eqref{irb8}, then imply that $w$ is a solution to the linear initial-value problem~\eqref{irb4} in the sense of \cite[Section~II.1.2]{Aman1995} and we infer from \cite[Remarks~II.2.1.2~(a)]{Aman1995} that $w$  has the representation formula 
	\begin{equation}
		w(t) =  e^{-\mu t}\tilde{U}(t,0) w^{in} + \int_0^t e^{-\mu(t-s)} \tilde{U}(t,s) F(s)\ \mathrm{d}s\,, \qquad t\in J_T\,.\label{irb9}
	\end{equation}
	It then follows from  \eqref{irb3}, \eqref{irb5}, and \eqref{irb9} that, for $t\in J_T$,
	\begin{align}
		\|w(t)\|_{W^{2\theta,p}} & \le K_{\theta} e^{(\omega-\mu) t} \|w^{in}\|_{W^{2\theta,p}} + K_{\theta} \int_0^t (t-s)^{-\theta} e^{(\omega-\mu)(t-s)} \|F(s)\|_p\ \mathrm{d}s \nonumber \\
		& \le C(p,\theta)  + C K_{\theta} \int_0^t (t-s)^{-\theta} e^{(\omega-\mu)(t-s)} \ \mathrm{d}s\,\\
		&\leq C(p,\theta)\,, \label{irb11}
	\end{align}
	since
	\begin{equation*}
		\int_0^\infty  s^{-\theta} e^{(\omega-\mu) s}\ \mathrm{d}s <\infty\,,	\end{equation*}
	and the proof is complete.
\end{proof}

\begin{proposition}\label{propLpu}
	There is  $r>3$ and $C>0$ depending on $r$ and the initial data but  independent of $T$ and $T_{\mathrm{max}}$\,, such that
	\begin{equation}
		\|\nabla v\|_r\leq C,\qquad t\in J_T.
	\end{equation}	
\end{proposition}
\begin{proof}
	We fix $p\in (2,3)$ and 
	\begin{equation*}
		\frac{1+p}{2p} < \theta' < \theta <1\,.
	\end{equation*} 
	From \eqref{defw} and Lemma~\ref{lem.impreg3}, we deduce that
	\begin{equation}
		\|u(t)\|_{W^{2\theta-2,p}} = \|\mA[w(t)]\|_{W^{2\theta-2,p}} \leq C \|w(t)\|_{W^{2\theta,p}} \leq C\,, \qquad t\in J_T\,. \label{y11}
	\end{equation}
	Introducing
	\begin{equation*}
		W_{\mathsf{B}}^{\xi,p}(\Omega) \triangleq \left\{ 
		\begin{array}{ll}
			\{ z \in W^{\xi,p}(\Omega)\ :\ \nabla z\cdot \mathbf{n} = 0 \;\;\text{ on}\;\; \partial\Omega \}\,, &  (p+1)/p < \xi \le 2\,, \\
			& \\
			W^{\xi,p}(\Omega)\,, & (1-p)/p < \xi < (p+1)/p \,,
		\end{array}
		\right.
	\end{equation*}
	see \cite[Section~7]{Aman1993}, and observing that the choice of $\theta$ and $p$ guarantees that $2\theta-2>(1-p)/p$, the realization in $W_{\mathsf{B}}^{2\theta-2,p}(\Omega)$ of $\mA$ generates an analytic semigroup in $W_{\mathsf{B}}^{2\theta-2,p}(\Omega)$,  see \cite[Theorem~8.5]{Aman1993}, and the interpolation space is characterized by
	\begin{equation*}
		\left( W_{\mathsf{B}}^{2\theta-2,p}(\Omega) , W_{\mathsf{B}}^{2\theta,p}(\Omega)\right)_{1+\theta'-\theta,p} = W_{\mathsf{B}}^{2\theta',p}(\Omega)
	\end{equation*}
	(up to equivalent norms). We then infer from \eqref{ks2}, \eqref{y2}, \eqref{y11} and \cite[Theorem~V.2.1.3]{Aman1995} that, for $t\in [t_*,T]$, 
	\begin{align*}
		\|v(t)\|_{W^{2\theta',p}} & \le \left\| e^{-\frac1\tau(t-t_*)\mA} v(t_*) \right\|_{W^{2\theta',p}} + \frac{1}{\tau}\int_{t_*}^t \left\| e^{-\frac1\tau(t-s)\mA} u(s) \right\|_{W^{2\theta',p}}\ \mathrm{d}s \\
		& \le C e^{- t/\tau} \|v(t_*)\|_{W^{2\theta',p}} + C \int_{t_*}^t (t-s)^{\theta-\theta'-1} e^{-(t-s)/\tau}\|u(s)\|_{W^{2\theta-2,p}}\ \mathrm{d}s \\
		& \le C \|v(t_*)\|_{W^{2,p}} + C \int_{0}^\infty s^{\theta-\theta'-1} e^{-s/\tau}\ \mathrm{d}s \\
		& \le C (1+ \|v(t_*)\|_{W^{2,p}})\,.
	\end{align*}
	As the choice of $\theta'$ and $p$ guarantees that $W^{2\theta',p}(\Omega)$ is continuously embedded in $W^{1,r}(\Omega)$ with some $r>3$, the above estimate implies that 
	\begin{equation}
		\|\nabla v(t)\|_r \le C\,, \qquad t\in [t_*,T]\,. \label{irb14}
	\end{equation}
	
With  the $W^{1,r}$-boundedness of $v$ in $[t_*,T]$ at hand, it becomes simple to derive the $L^q$-boundedness of $u$ with any $q>1$. To see this, let us multiply \eqref{ks1} by $u^{q-1}$, use \eqref{f1a} and integrate by parts to obtain that
\begin{equation}
	\begin{split}
		\frac{1}{q}\frac{\rd}{\rd t}\int_\Omega u^q\ \rd x & +(q-1)\int_\Omega e^{-v}u^{q-2}|\nabla u|^2\ \rd x+a_1\int_\Omega u^q\ \rd x \\
		& \leq(q-1)\int_\Omega e^{-v}u^{q-1}\nabla u\cdot\nabla v\ \rd x+b_1\int_\Omega u^{q-1}\,\rd x\,. 
	\end{split}\label{uLq0}
\end{equation}
By Young's inequality,  we infer that, for $t\in[t_*,T]$,
\begin{align*}
	&(q-1)\int_\Omega e^{-v}u^{q-1}\nabla u\cdot\nabla v\,\rd x+b_1\int_\Omega u^{q-1}\,\rd x\\
	&\leq \frac{q-1}{2}\int_\Omega e^{-v}u^{q-2}|\nabla u|^2\,\rd x+\frac{q-1}{2}\int_\Omega e^{-v}u^q|\nabla v|^2\ \rd x+\frac{a_1}{2}\int_\Omega u^q\ \rd x+C(a_1,b_1,q)\\
	&\leq \frac{q-1}{2}\int_\Omega e^{-v}u^{q-2}|\nabla u|^2\,\rd x + C(q)\int_\Omega u^q|\nabla v|^2\,\rd x+\frac{a_1}{2}\int_\Omega u^q\ \rd x+C(a_1,b_1,q)\,.
\end{align*}
On the one hand, thanks the lower boundedness of $e^{-v}$ by \eqref{bg},
\begin{equation*}
	\frac{q-2}{2}\int_\Omega e^{-v}u^{q-2}\,\rd x\geq c(q)\|\nabla u^{q/2}\|^2
\end{equation*}
On the other hand, we use \eqref{irb14}, H\"older's inequality, and Young's inequality to deduce that, for $t\in[t_*,T]$,
\begin{align*}
	\int_\Omega u^q|\nabla v|^2\,\rd x\leq &\|u^{q/2}\|_{\frac{2r}{r-2}}^2\|\nabla v\|_r^2\\
	\leq& C\|u^{q/2}\|_{\frac{2r}{r-2}}^2\\
	\leq &C\|u^{q/2}\|_6^{2\alpha}\|u\|_1^{2(1-\alpha)}\\
	\leq& C\|\nabla u^{q/2}\|_2^{2\alpha}\|u\|_1^{2(1-\alpha)}+C\|u\|_1^2\,\\
	\leq&\varepsilon \|\nabla u^{q/2}\|_2^2+C(\varepsilon)\|u\|_1^2
\end{align*}with any $\varepsilon>0$, since here 
\begin{equation}
	0<\alpha:=\frac{3qr-3r+6}{3qr-r}<1.
\end{equation}
Collecting the above estimates and pick $\varepsilon$ small, we arrive at
\begin{equation}
	\frac{1}{q}\frac{\rd}{\rd t}\int_\Omega u^q\ \rd x  +\frac{a_1}{2}\int_\Omega u^q\ \rd x \leq C(q)\,. 
\label{uLq1}	
\end{equation}
Solving the above differential  inequality then provides the uniform boundedness of $\|u\|_q$ in $[t_*,T]$, while that on $[0,t_*]$ is a consequence of Proposition~\ref{prop2.1}.
\end{proof}

\begin{proof}[Proof of Theorem~\ref{Th1} in 3D]
With the aid of Proposition~\ref{propLpu}, we may further use a standard bootstrap argument to prove that (see, e.g., \cite[Lemma A.1]{TW2012}), for any $0<T<T_{\mathrm{max}}$, there is $C>0$ independent of $T_{\mathrm{max}}$ and $T$ such that
\begin{equation*}
	\sup\limits_{0\leq t\leq T}\|u(t)\|_{\infty}\leq C\,.	
\end{equation*}	
	According to Proposition~\ref{prop2.1}, we deduce that $T_{\mathrm{max}}= \infty$ and thus Theorem~\ref{Th1} is proved.
\end{proof}

\section*{Acknowledgments}
Xiao acknowledges the support of Science Foundation of Hebei Normal
University (No. ~L2024B03). Jiang is supported by National Natural Science Foundation of China (NSFC)
under grants No.~12271505 \& No.~12071084,  the Training Program of Interdisciplinary Cooperation of Innovation Academy for Precision Measurement Science and Technology, CAS (No.~S21S3202), and by Knowledge Innovation Program of Wuhan-Basic Research (No.~2022010801010135).  Jiang thanks Prof. Philippe Lauren\c cot for inspiring discussions with him.
\bibliographystyle{siam}
\bibliography{PSE2023}

\end{document}